\documentclass[11pt]{article}
\usepackage{latexsym,amsmath,amsthm,amssymb,color}
\usepackage{mathrsfs,wasysym,graphicx,lscape}

\begin{document}
\newtheorem{theorem}{\indent Theorem}[section]
\newtheorem{proposition}[theorem]{\indent Proposition}
\newtheorem{definition}[theorem]{\indent Definition}
\newtheorem{lemma}[theorem]{\indent Lemma}
\newtheorem{remark}[theorem]{\indent Remark}
\newtheorem{corollary}[theorem]{\indent Corollary}

%%%%%%%%%%
\begin{center}
    {\large \bf  On the blow-up solutions for the fractional nonlinear Schr\"{o}dinger equation with combined power-type nonlinearities}
\vspace{0.5cm}\\{ Binhua Feng}\\
{\small Department of Mathematics, Northwest Normal University, Lanzhou, 730070, P.R. China }\\
\end{center}

\renewcommand{\theequation}{\arabic{section}.\arabic{equation}}
\numberwithin{equation}{section}
\footnote[0]{\hspace*{-7.4mm}
%%%%%%%%%%
E-mail: binhuaf@163.com(Binhua Feng)\\
This work is supported by NSFC Grants (No. 11601435, No. 11401478), Gansu Provincial Natural
Science Foundation (1606RJZA010) and NWNU-LKQN-14-6.}

\renewcommand{\baselinestretch}{1.6}
\large\normalsize
\begin{abstract}
This paper is devoted to the analysis of blow-up solutions for the fractional nonlinear Schr\"{o}dinger equation with combined power-type nonlinearities
\[
i\partial_t u-(-\Delta)^su+\lambda_1|u|^{2p_1}u+\lambda_2|u|^{2p_2}u=0,
\]
where $0<p_1<p_2<\frac{2s}{N-2s}$.
Firstly, we obtain some sufficient conditions about existence of blow-up solutions, and then derive some sharp thresholds of blow-up and global existence by constructing some new estimates. Moreover, we find the sharp threshold mass
 of blow-up and global existence in the case $0<p_1<\frac{2s}{N}$ and $p_2=\frac{2s}{N}$.
Finally, we investigate the dynamical properties of blow-up solutions, including $L^2$-concentration, blow-up rate and limiting profile.

%2010 Mathematics Subject Classification: 35Q55, 35A15.
\end{abstract}
\section{Introduction}

In recent years, there has been a great deal of interest in using fractional Laplacians to model physical phenomena. By extending the
Feynman path integral from the Brownian-like to the L\'{e}vy-like quantum mechanical
paths, Laskin in \cite{n1,n2} used the theory of functionals over functional measure generated
by the L\'{e}vy stochastic process to deduce the following nonlinear fractional Schr\"{o}dinger
equation
\begin{equation}\label{e0}
i\partial_{t}u=(-\Delta)^{s}u+f(u),
\end{equation}
where $0<s<1$, $f(u)=|u|^{2p}u$. The fractional differential
operator $(-\Delta)^{s}$ is defined by $(-\Delta)^{s}u=\mathcal{F}^{-1}[|\xi|^{2s}\mathcal{F}(u)]$, where $\mathcal{F}$ and $\mathcal{F}^{-1}$
are the Fourier transform and inverse Fourier transform, respectively.

Recently, equation \eqref{e0} has attracted more and more attentions in both the physics and mathematics fields, see \cite{y1,y2,y3,y4,y5,f18jmaa,f18cma,gzjde,hong,j2,zjde,zjee}.
For the Hartree-type nonlinearity $(|x|^{-\gamma}\ast|u|^2)u$, Cho et al. in  \cite{y1}  proved existence and uniqueness of local and global solutions of \eqref{e0}. They also showed the existence of blow-up solutions in \cite{y4}. The dynamical properties of blow-up solutions have been investigated in \cite{y3,zjde}. Zhang and Zhu in \cite{j2} studied the stability and instability of standing waves. For the local nonlinearity $|u|^{2p}u$, the well-posedness and ill-posedness in the Sobolev space $H^s$ have been investigated in \cite{y5,hong}.
In \cite{bou}, Boulenger et al. have obtained a general criterion for blow-up of radial solution of \eqref{e0} with $p\geq\frac{2s}{N}$  in $\mathbb{R}^{N}$ with $N\geq2$. Although a general existence theorem for blow-up solutions of this problem has remained an open problem, it has been strongly supported by numerical evidence \cite{ksm}. The orbitally stability of standing waves for other kinds of fractional Schr\"{o}dinger equations has been studied in \cite{f18jmaa,f18cma,y2,zjee}.

In this paper, we consider the following fractional nonlinear Schr\"{o}dinger equation with combined power-type nonlinearities
\begin{equation}\label{e}
\left\{
\begin{array}{l}
i\partial_tu-(-\Delta)^su+\lambda_1|u|^{2p_1}u+\lambda_2|u|^{2p_2}u=0, \\
u(0,x) = u_0 (x),%
\end{array}%
\right.
\end{equation}
where  $u=u(t,x):[0,T^*)\times \mathbb{R}^N \rightarrow \mathbb{C}$ is a complex valued function, $0<s<1$, $\lambda_1,\lambda_2\in \mathbb{R}$, $0<p_1<p_2<\frac{2s}{N-2s}$.
This equation has Hamiltonian
\begin{align}\label{h}
E(u(t))=&\frac{1}{2}\int_{\mathbb{R}^N} |(-\Delta)^{s/2} u(t,x)|^2 dx-\frac{\lambda_1}{2p_1+2}\int_{\mathbb{R}^N}
|u(t,x)|^{2p_1+2}dx \nonumber \\&-\frac{\lambda_2}{2p_2+2}\int_{\mathbb{R}^N}
|u(t,x)|^{2p_2+2}dx.
\end{align}
But there
is no scaling invariance for this equation.

When $s=1$ and $\lambda_2=0$, equation \eqref{e} reduces the following classical nonlinear Schr\"{o}dinger equation
\begin{equation}\label{0}
i\partial_tu+\Delta u=\lambda_1 |u|^{2p_1}u.
\end{equation}
Because of important applications in physics, nonlinear Schr\"{o}dinger
equations received a great deal of attention from mathematicians in the past decades, see \cite{ca2003,ss,t} for a review.
Ginibre and Velo \cite{gv} established the local well-posedness of \eqref{0} in $H^1$(
see \cite{ca2003} for a review). When $\lambda_1<0$ and $\frac{2}{N}\leq p_1\leq \frac{2}{N-2}$, Glassey \cite{gl} proved the existence of blow-up solutions
 for the negative energy and $|x|u_0\in L^2$. Ogawa and Tsutsumi \cite{ot} proved the existence
  of blow-up solutions in radial case without the restriction $|x|u_0\in L^2$. A natural question appears for $p_1\geq\frac{2}{N}$
: can one find
some sharp criteria for  blow-up and global existence of \eqref{0}? Weinstein \cite{we} gave a crucial criterion in terms of
$L^2$-mass initial data. Also, some sharp criteria in terms of the energy of the initial data were obtained (see \cite{lz,zj}). Cazenave also mentioned this topic in their monographs \cite{ca2003}. From the view point of physics, this problem is also
pursued strongly (see \cite{ku} and the references therein). In addition, for the $L^2$-critical nonlinearity, i.e., $p_1=\frac{2}{N}$, Weinstein
 \cite{we1} studied the structure and formation of singularity of blow-up solutions with
 critical mass by the concentration compact principle: the blow-up solution is close to
 the ground state in $H^1$ up to scaling and phase parameters, and also translation in
 the non-radial case.
%We see that the blow-up solution has the same shape as the ground state solution.
%Then, Merle and Tsutsumi [24,36] showed the $L^2$-concentration phenomenon of blow-up solutions.
% Merle [25,26]constructed the exact blow-up solutions with critical mass by the conformal invariance and compact results.
Applying the variational methods, Merle and Rapha\"{e}l \cite{m2} improved Weinstein's results and obtained the sharp decomposition of blow-up solutions with small super-critical mass. By this sharp decomposition and spectral properties, Merle and Rapha\"{e}l \cite{m1,m2,m3,m4} obtained a large body of breakthrough works, such as sharp blow-up rates, profiles, etc. Hmidi and Keraani \cite{ke} established the profile decomposition of bounded sequences in $H^1$ and gave a new and simple proof of some dynamical properties of
blow-up solutions in
$H^1$. These results have been generalized to other kinds of Schr\"{o}dinger equations, see \cite{fdsrwa,feect,f14jmaa,lz,zz,zjde,zjmp}.

In \cite{tao}, Tao et al. undertook a comprehensive study for the following nonlinear Schr\"{o}dinger
equation with combined power-type nonlinearities
\begin{equation}\label{taoe}
\left\{
\begin{array}{l}
i\partial_tu+\Delta u+\lambda_1|u|^{2p_1}u+\lambda_2|u|^{2p_2}u=0, \\
u(0,x) = u_0 (x),%
\end{array}%
\right.
\end{equation}
where $0<p_1<p_2\leq \frac{2}{N-2}$.
More precisely, they addressed questions
related to local and global well-posedness, finite time blow-up, and asymptotic
behaviour. Recently, in \cite{fjee}, we prove the existence of blow-up solutions and find the sharp
threshold mass of blow-up and global existence for \eqref{taoe} with $p_1= \frac{2}{N}$ and $0<p_2< \frac{2}{N}$, which is a complement to the result in \cite{tao}.

As far as we know, the existence of blow-up solutions of \eqref{e} has not been proved yet.
In particular, the dynamical properties of
blow-up solutions have not been proved even when $\lambda_1=0$.
In this paper, we will focus on the blow-up solutions of \eqref{e}. More precisely, we
are interested in sufficient conditions about the existence of blow-up solutions, sharp thresholds of blow-up and global existence, the dynamical properties of
blow-up solutions, including $L^2$-concentration, blow-up rates, and limiting profile.

To solve these problems, we mainly use the ideas from Boulenger et al. \cite{bou}
and Keraani \cite{ke}. The existence of blow-up solutions for the fractional nonlinear
Schr\"{o}dinger equation \eqref{e0} with the local nonlinearity $|u|^{2p}u$ has been investigated in \cite{bou}. The dynamical properties of blow-up solutions for the $L^2$-critical nonlinear Schr\"{o}dinger equation \eqref{0} have been discussed in \cite{ke}. In these papers, the study of blow-up solutions relies heavily on the scaling invariance of \eqref{e0} and \eqref{0}.
 Hence, the study of blow-up solutions for \eqref{e}, which has no the scaling invariance, is of particular interest.

 Firstly, we will investigate sufficient conditions about the existence of blow-up solutions for \eqref{e} by using the method of Boulenger et al.. In addition, in \cite{bou}, they use $E(u)^{s_c}M(u)^{s-s_c}$ and $\|(-\Delta)^{\frac{s}{2}}u\|^{s_c}_{L^{2}}
\| u \|^{s-s_c}$ to obtain some sharp thresholds of blow-up in finite time, where $s_c=\frac{N}{2}-\frac{s}{p}$. Note that the quantities $E(u)^{s_c}M(u)^{s-s_c}$ and $\|(-\Delta)^{\frac{s}{2}}u\|^{s_c}_{L^{2}}
\| u \|^{s-s_c}$ are scaling invariant of \eqref{e0}. But there
is no scaling invariance for equation \eqref{e}. Therefore, we must construct some new estimates to obtain some sharp thresholds of blow-up and global existence.

When $0<p_1<\frac{2s}{N}$ and $p_2=\frac{2s}{N}$, by using the scaling argument and the variational characteristic provided by the
sharp Gagliardo-Nirenberg inequality \eqref{gn}, we find the sharp threshold mass
$\|Q\|_{L^2}$ of blow-up and global existence for \eqref{e} in the following sense, where $Q$ is the ground state
solution of \eqref{ell} with $p=\frac{2s}{N}$.

(i) If $\|u_0\|_{L^2}<\|Q\|_{L^2}$, then the solution of \eqref{e} exists globally in $H^s$.

(ii) If $\|u_0\|_{L^2}\geq\|Q\|_{L^2}$, we can construct a class of initial data, and the corresponding solution $u(t)$ of \eqref{e} must blow up.

Finally, in order to overcome the loss of scaling invariance, we use the ground state solution $Q$ of \eqref{ell}
to describe the dynamical behaviour of the blow-up solutions to \eqref{e} with $0<p_1<\frac{2s}{N}$ and $p_2=\frac{2s}{N}$, including $L^2$-concentration, blow-up rates, and limiting profile. Our method can be easily applied to study the dynamical behaviour of the blow-up solutions to \eqref{e} with $\lambda_1=0$ and $p_2=\frac{2s}{N}$. Our results are new even for \eqref{e} with $\lambda_1=0$ and $p_2=\frac{2s}{N}$.

This paper is organized as follows: in Section 2, we present some preliminaries. In section 3, we will establish some sufficient conditions of the existence of blow-up solutions for \eqref{e}, and then obtain some sharp thresholds of blow-up and global existence. Moreover, we find the sharp threshold mass
 of blow-up and global existence for \eqref{e}. In section 4, we will consider some dynamical properties of
blow-up solutions of \eqref{e} with $p_2=\frac{2s}{N}$ and $0<p_1<\frac{2s}{N}$, including $L^2$-concentration, blow-up rate, and limiting profile.

\textbf{Notation.}
Throughout this paper, we use the following
notation. $C> 0$ will stand for a constant that may be different from
line to line when it does not cause any confusion.
We often abbreviate $L^q(\mathbb{R}^N)$, $\|\cdot\|_{L^q(\mathbb{R}^N)}$ and $H^s(\mathbb{R}^N)$ by $L^q$, $\|\cdot\|_{L^q}$ and $H^s$, respectively.
%For $s \in \mathbb{R}$, the pseudo-differential operator $(-\Delta)^{s}$ is defined by $\widehat{(-\Delta)^{s}}f(\xi)=|\xi|^{2s}\hat{f}(\xi)$, where $\hat{}$ denotes Fourier transform.
%We also use
%homogeneous Sobolev space $\dot{H}^s(\mathbb{R}^{N})=\{u \in
%\mathcal{S}^\prime(\mathbb{R}^{N}) ; \int |\xi|^{2s}|\hat{f}(\xi)|^2d\xi<\infty \}$ with its norm defined by $\|f\|_{\dot{H}^s}=\|(-\Delta)^{s/2}f\|_{L^2}$, where $\mathcal{S}^\prime(\mathbb{R}^{N})$
%denotes the space of tempered distribution on $\mathbb{R}^{N}$.
%In particular, we use notation: $s_c=\frac{N}{2}-\frac{s}{p_2}$ and $p_c=\frac{Np_2}{s}$. Therefore, it follows from  Sobolev imbedding that $\dot{H}^{s_c}(\mathbb{R}^N)\hookrightarrow L^{p_c}(\mathbb{R}^N)$.

\section{Preliminaries}
Firstly, by a similar argument as that in \cite{y5,hong}, we can establish the local theory for the Cauchy problem \eqref{e}, see also \cite{zjee}.
\begin{proposition}
Let $u_0 \in H^{s}$ and $0<p_1<p_2<\frac{2s}{N-2s}$. Then, there exists $T = T(\|u_0
\|_{H^{s}})$ such that \eqref{e} admits a unique solution $u\in C([0,T],H^s)$. Let $[0,T^{\ast })$ be the maximal time interval on which the
solution $u$ is well-defined, if $T^{\ast }< \infty $, then $\|
u(t)\| _{H^{s}}\rightarrow \infty $ as $t\uparrow T^{\ast } $. Moreover, for all $0\leq t<T^*$, the solution
$ u(t)$ satisfies the following conservation of mass and energy
\[
\|u(t)\|_{L^2}=\|u_0\|_{L^2},
\]
\[
E(u(t))=E(u_0 ),
\]
 where $E(u(t))$ defined by \eqref{h}.
\end{proposition}

Next,
we recall a sharp Gagliardo-Nirenberg type inequality established in \cite{bou,zjee}.
\begin{lemma}
Let $N\geq 2$, $0<s<1$ and $0<p<\frac{2s}{N-2s}$. Then, for all $u\in H^s$,
\begin{equation}\label{gn}
\int_{\mathbb{R}^{N}}|u|^{2p+2}dx\leq C_{opt}\|(-\Delta)^{\frac{s}{2}}u\|^{\frac{p N}{s}}_{L^{2}}
\| u \|^{(2p+2)-\frac{pN}{s}}_{L^{2}},
\end{equation}
where the optimal constant $C_{opt}$ given by
\[
 C_{opt}=\left(\frac{2s(p+1)-p N}{p N}\right)^{\frac{Np}{2s}}\frac{2s(p+1)}{(2s(p+1)-p N)\| Q \|_{L^{2}}^{2p}},
\]
and $Q$ is a ground state solution of
\begin{equation}\label{ell}
(-\Delta)^{s}Q+Q=|Q|^{2p}Q~~~in~~\mathbb{R}^N.
\end{equation}
In particular, in the $L^2$-critical case $p=\frac{2s}{N}$, $ C_{opt}=\frac{p+1}{\| Q \|_{L^{2}}^{2p}}$.

Moreover, the solution $Q$ satisfies the following relations
\begin{equation} \label{e2.3}
 \|(-\Delta)^{\frac{s}{2}} Q\|_{L^{2}}^2=\frac{p N}{2s(p+1)-pN}
\| Q\|_{L^{2}}^2,
\end{equation}
 and
 \begin{equation} \label{e2.4}
  \int_{\mathbb{R}^{N}}|Q|^{2p+2}dx=\frac{2s(p+1)}{2s(p+1)-p N}\|Q\|_{L^{2}}^{2}.
\end{equation}
\end{lemma}

Next, we shall recall the profile decomposition of bounded sequences
in $H^s$, which is important
to study the dynamical properties of blow-up solutions, see \cite{zjee}.
\begin{proposition}
Let $N\geq 2$ and $0<s<1$.
 Assume that $\{v_n\}_{n=1}^{\infty}$ is a bounded sequence in $H^s$. Then, there exist a subsequence of $\{v_n\}_{n=1}^{\infty}$ (still denoted by $\{v_n\}_{n=1}^{\infty}$), a family $\{x^j\}_{j=1}^{\infty}$ of sequences in $\mathbb{R}^N$ and a sequence $\{V^j\}_{j=1}^{\infty}$ in $H^s$ such that

(i) for every $k\neq j$, $|x_n^k-x_n^j|\rightarrow +\infty$, as $n\rightarrow \infty$;

(ii) for every $l\geq 1$ and every $x\in \mathbb{R}^N$, it follows

\begin{equation}\label{3.a}
v_n(x)=\sum_{j=1}^{l}V^j(x-x_n^j)+v_n^l(x),
\end{equation}
with
\[
\limsup_{n\rightarrow \infty}\|v_n^l\|_{L^q}\rightarrow 0 ~as ~l\rightarrow \infty
\]
for every $q \in (2,\frac{2N}{N-2s})$. Moreover, we have, as $n\rightarrow \infty$,
\begin{equation}\label{3.b}
\|v_n\|_{\dot{H}^s}^2=\sum_{j=1}^{l}\|V^j\|_{\dot{H}^s}^2+\|v_n^l\|_{\dot{H}^s}^2+\circ(1),
\end{equation}
\begin{equation}\label{3.c}
\int_{\mathbb{R}^N} |\sum_{j=1}^l V^j(x -x_n^j)|^q dx=\sum_{j=1}^l\int_{\mathbb{R}^N} |V^j(x -x_n^j)|^q dx +\circ(1),
\end{equation}
where $\circ(1):=\circ_n(1)\rightarrow 0$ as $n\rightarrow \infty$.
\end{proposition}
\textbf{Remark.}  In this proposition, the number of non-zero terms in the right side of \eqref{3.a} may be one, finite and infinite, which may correspond to three possibilities (compactness, dichotomy and vanishing) in the concentration compactness principle proposed by Lions. Hence, the profile decomposition may look as another equivalent description of the concentration compactness principle. However, there are two major advantages of the profile decomposition of bounded sequences in $H^s$: one is that the decomposing expression of the bounded sequence $\{v_n\}_{n=1}^{\infty}$ is given and we can inject it into our aim functionals, and the other is that the decomposition is orthogonal by (i) and norms of $\{v_n\}_{n=1}^{\infty}$ have similar decompositions, for example \eqref{3.b}. Those properties are useful in the calculus of variational methods.

In this paper, we will use the method in \cite{bou} to prove the existence of blow-up solutions to \eqref{e}. In the following, we recall some important results in \cite{bou}.
\begin{lemma}\cite{bou}  \label{Lemma3.1}\rm
 Let $N\geq1$ and suppose $\varphi:\mathbb{R}^{N}\rightarrow\mathbb{R}$ is such that $\nabla\varphi\in W^{1,\infty}(\mathbb{R}^{N})$. Then, for all $u\in H^{\frac{1}{2}}(\mathbb{R}^{N})$, it holds that
\[
 \left|\int_{\mathbb{R}^{N}}\overline{u}(x)\nabla\varphi(x)\cdot\nabla u(x)dx\right|\leq C(\||\nabla|^{\frac{1}{2}}u\|^{2}_{L^{2}}+\| u\|_{L^{2}}\||\nabla|^{\frac{1}{2}}u\|_{L^{2}}),
\]
 with some constant $C>0$ that depends only on $\|\nabla\varphi\|_{W^{1,\infty}}$ and $N$.
\end{lemma}
\begin{lemma} \cite{bou} \label{Lemma3.2}\rm
 Let $N\geq1$, $s\in(0,1)$ and suppose $\varphi:\mathbb{R}^{N}\rightarrow\mathbb{R}$ with $\Delta\varphi\in W^{2,\infty}(\mathbb{R}^{N})$. Then, for all $u\in L^{2}(\mathbb{R}^{N})$, we have
\[
 \left|\int_{0}^{\infty}m^{s}\int_{\mathbb{R}^{N}}(\Delta^{2}\varphi)| u_{m}|^{2}dxdm\right|\leq C\|\Delta^{2}\varphi\|^{s}_{L^{\infty}}\|\Delta\varphi\|^{1-s}_{L^{\infty}}\| u\|^{2}_{L^{2}}.
\]
\end{lemma}

Let us assume that $\varphi:\mathbb{R}^{N}\rightarrow \mathbb{R}$ is a real-valued function with $\nabla\varphi\in W^{3,\infty}(\mathbb{R})$. We define the localized virial of $u=u(t,x)$ to be the quantity given by
\begin{equation} \label{y1}
  \mathcal{M}_{\varphi}[u(t)]:=2Im\int_{\mathbb{R}^{N}}\bar{u}(t)\nabla \varphi\cdot\nabla u(t)dx.
\end{equation}
 By applying Lemma 2.4, we obtain the bound
\[
\mid \mathcal{M}_{\varphi}[u(t)]\mid\leq C(\|\nabla\varphi\|_{L^{\infty}},\|\Delta\varphi\|_{L^{\infty}})\| u(t)\|_{H^{\frac{1}{2}}}^{2}.
\]
Hence the quantity $\mathcal{M}_{\varphi}[u(t)]$ is well-defined, since $u(t)\in H^{s}(\mathbb{R}^{N})$ with some $s\geq\frac{1}{2}$ by assumption.

To study the time evolution of $\mathcal{M}_{\varphi}[u(t)]$, we shall need the following auxiliary function
\begin{equation} \label{y2}
  u_{m}(t):=c_{s}\frac{1}{-\Delta+m}u(t)=c_{s}\mathcal{F}^{-1}
  \left(\frac{\hat{u}(t,\xi)}{|\xi|^{2}+m}\right),
~~~with~~~m>0,
\end{equation}
where the constant
\begin{equation*}\label{y3}
  c_{s}:=\sqrt{\frac{\sin \pi s}{\pi}}
\end{equation*}
turns out to be a convenient normalization factor. By the smoothing properties of $(-\Delta+m)^{-1}$, we clearly have that $u_{m}(t)\in H^{\alpha+2}(\mathbb{R}^{N})$ holds for any $t\in[0,T^*)$ whenever  $u(t)\in H^{\alpha}(\mathbb{R}^{N})$.

By a similar argument as that in \cite{bou}, we have the following time evolution of $\mathcal{M}_{\varphi}[u(t)]$.
\begin{lemma}  \label{Lemma3.3}\rm
 For any $t\in[0,T^*)$, we have the identity
\begin{align}\label{c4}
 \frac{d}{dt}\mathcal{M}_{\varphi}[u(t)]=&\int_{0}^{\infty}m^{s}\int_{\mathbb{R}^{N}}\{4\overline{\partial_{k}u_{m}}(\partial^{2}_{kl}\varphi)\partial_{l}u_{m}
 -(\Delta^{2}\varphi)| u_{m}|^{2}\}dxdm\nonumber\\
 &-\frac{2\lambda_1p_1}{p_1+1}\int_{\mathbb{R}^{N}}|u|^{2p_1+2}\Delta\varphi dx-
\frac{2\lambda_2p_2}{p_2+1}\int_{\mathbb{R}^{N}}|u|^{2p_2+2}\Delta\varphi dx
 \end{align}
 where $u_{m}=u_{m}(t,x)$ is defined in \eqref{y2} above.
\end{lemma}
%\textbf{Remark.}
%From the definition of $u_{m}$, we have
%\begin{align*}
%&\hskip0.4cm\int_{0}^{\infty}m^{s}\int_{\mathbb{R}^{N}}|\nabla u_{m}|^{2}dxdm=\int_{\mathbb{R}^{N}}(\frac{\sin \pi s}{\pi}\int_{0}^{\infty}\frac{m^{s}dm}{(|\xi|^{2}+m)^{2}})|\xi|^{2}|\widehat{u}(\xi)|^{2}d\xi\\
%&=\int_{\mathbb{R}^{N}}
%(s|\xi|^{2s-2})|\xi|^{2}|\widehat{u}(\xi)|^{2}d\xi=s\|(-\Delta)^{\frac{s}{2}}u(t)\|^{2}_{L^{2}}
%\end{align*}
%for arbitrary $u\in \dot{H}^{s}(\mathbb{R}^{N})$. It's important for our proof.

%\textit{For the time evolution of the localized virial $\mathcal{M}_{\varphi}[u(t)]$ with $\varphi_{R}$ as above, we have the following estimate.
%}

Let $\varphi:\mathbb{R}^{N}\rightarrow\mathbb{R}$ be as above. In addition, we assume that $\varphi=\varphi(r)$ is radial and satisfies
$$\varphi(r)=
\begin{cases}
 \frac{r^{2}}{2}\,\,\,\,\,\mbox{for}\,\,\,\,\,r\leq1,\\
 const.\,\,\,\,\,\mbox{for}\,\,\,\,\,r\geq10,
\end{cases}
$$
and $\varphi''(r)\leq1$ for $r\geq0$. Given $R>0$ , we define the rescaled function $\varphi_{R}:\mathbb{R}^{N}\rightarrow\mathbb{R}$ by
\[
\varphi_{R}(r):=R^{2}\varphi(\frac{r}{R}).
\]
We readily verify the inequalities
\[
1-\varphi_{R}''(r)\geq0,~~~1-\frac{\varphi_{R}'(r)}{r}\geq0,~~~N-\Delta\varphi_{R}(r)\geq0,
\]
for all $r\geq0$.
%Indeed, this first inequality follows from $\varphi_{R}''(r)=\varphi''(\frac{r}{R})\leq1$. We obtain the second inequality by integrating the first inequality on $[0,r]$ and using that $\varphi_{R}'(0)=0$. For the third inequality, we find that $N-\Delta\varphi_{R}(r)=1+\varphi_{R}''(r)+(N-1)(1-\frac{1}{r}\varphi_{R}'(r))\geq0$ holds thanks to the first two inequalities. Finally, we noticed that$$\frac{x\cdot\nabla\varphi_{R}(r)}{|x|^{2}}=\frac{R\varphi_{R}'(\frac{r}{R})| x|}{|x|^{2}}=\frac{R}{r}\varphi_{R}'(\frac{r}{R}).$$
%So, $1-\frac{x\cdot\nabla\varphi_{R}(r)}{|x|^{2}}\geq0$ due to $\varphi_{R}'(r)\leq r$.

By  a similar argument as Lemma 2.2 in \cite{bou}, we obtain the following time evolution of the localized virial $\mathcal{M}_{\varphi_R}[u(t)]$ with $\varphi_R$ as above.
\begin{lemma}  \label{Lemma3.5}\rm (Localized radial virial estimate)
Let $N\geq2,s\in(\frac{1}{2},1)$ and assume in addition that $u(t)$ is a radial solution of \eqref{e}. We then have
 \begin{align}\label{vir}
 \frac{d}{dt}\mathcal{M}_{\varphi_R}[u(t)]\leq&
 4s\|(-\Delta)^{\frac{s}{2}}u(t)\|^{2}_{L^{2}}
 -\frac{2\lambda_1Np_1}{p_1+1}\|u(t)\|_{2p_1+2}^{2p_1+2}-\frac{2\lambda_2Np_2}{p_2+1}
 \|u(t)\|_{2p_2+2}^{2p_2+2}\nonumber\\
 +C(R^{-2s}+&R^{-p_1(N-1)+\varepsilon_1 s}\|(-\Delta)^{\frac{s}{2}}u(t)\|^{\frac{p_1}{s}+\varepsilon_1}_{L^{2}}+R^{-p_2(N-1)+\varepsilon_2 s}\|(-\Delta)^{\frac{s}{2}}u(t)\|^{\frac{p_2}{s}+\varepsilon_2}_{L^{2}})\nonumber\\=4p_2NE(u_{0})&-2(p_2 N-2s)\|(-\Delta)^{\frac{s}{2}}u(t)\|^{2}_{L^{2}}+\frac{2\lambda_1N(p_2-p_1)}{p_1+1}\|u(t)\|_{2p_1+2}^{2p_1+2}\nonumber\\
 +C(R^{-2s}+&R^{-p_1(N-1)+\varepsilon_1 s}\|(-\Delta)^{\frac{s}{2}}u(t)\|^{\frac{p_1}{s}+\varepsilon_1}_{L^{2}}+R^{-p_2(N-1)+\varepsilon_2 s}\|(-\Delta)^{\frac{s}{2}}u(t)\|^{\frac{p_2}{s}+\varepsilon_2}_{L^{2}}),
  \end{align}
  for any $0<\varepsilon_1<\frac{p_1(2s-1)}{s}$ and $0<\varepsilon_2<\frac{p_2(2s-1)}{s}$. Here $C=C(\| u_{0}\|_{L^{2}},N,\varepsilon_1,\varepsilon_2,s,p_1,p_2)$ is some positive constant.
 \end{lemma}

In order to deal with the $L^{2}$-critical case, we shall need the following refined version of Lemma 2.7 involving the nonnegative radial functions
\[
\psi_{1,R}=1-\varphi_{R}''(r)\geq0,\,\,\,\,\psi_{2,R}=N-\Delta\varphi_{R}(r)\geq0.
\]

\begin{lemma}  \label{Lemma4.1}(A Refined Version of Lemma 2.7)
Let $N\geq2,s\in(\frac{1}{2},1)$ and assume in addition that $u(t)$ is a radial solution of \eqref{e} for any $t\in[0,T^*)$ and $p_2=\frac{2s}{N}$. We then have
\begin{align}\label{h1}
\frac{d}{dt}\mathcal{M}_{\varphi}[u(t)]\leq &8sE[u_{0}]-4\int_{0}^{\infty}m^{s}\int_{\mathbb{R}^{N}}(\psi_{1,R}-C(\eta)\psi_{2}^{\frac{N}{2s}})|\nabla u_{m}|^{2}dxdm\nonumber \\&+\frac{2\lambda_1N(p_2-p_1)}{p_1+1}\int_{\mathbb{R}^N}
|u(t,x)|^{p_1+2}dx+CR^{-p_1(N-1)+\varepsilon_1 s}\|(-\Delta)^{\frac{s}{2}}u(t)\|^{\frac{p_1}{s}+\varepsilon_1}_{L^{2}}\nonumber \\&+\mathcal{O}((1+\eta^{-\beta})R^{-2s}+\eta(1+R^{-2}+R^{-4}))
\end{align}
%\begin{align}\label{h1}
%\frac{d}{dt}\mathcal{M}_{\varphi_R}[u(t)]\leq &8sE[u_{0}]-4\int_{0}^{\infty}m^{s}\int_{\mathbb{R}^{N}}(\psi_{1,R}-C(\eta)\psi_{2,R}^{\frac{N}{2s}})|\nabla u_{m}|^{2}dxdm\nonumber \\&-\frac{2N(p_2-p_1)}{p_1+1}\int_{\mathbb{R}^N}
%|u(t,x)|^{p_1+2}dx+CR^{-p_1(N-1)+\varepsilon_1 s}\|(-\Delta)^{\frac{s}{2}}u(t)\|^{\frac{p_1}{s}+\varepsilon_1}_{L^{2}}\nonumber \\&+\mathcal{O}((1+\eta^{-\beta})R^{-2s}+\eta(1+R^{-2}+R^{-4}))
%\end{align}
 for any $\eta>0$ and $R>0$, $0<\varepsilon_1<\frac{p_1(2s-1)}{s}$, where $C(\eta)=\frac{\eta}{N+2s}$ and $\beta=\frac{2s}{N-2s}$.
 \end{lemma}

\section{The existence of blow-up solutions}
In this section, we will establish some sufficient conditions about the existence of blow-up solutions for \eqref{e}, and then obtain some sharp thresholds of blow-up and global existence. Moreover, we find the sharp threshold mass
 of blow-up and global existence for \eqref{e}.
Firstly, we will prove the existence of blow-up solutions of \eqref{e}.
\begin{theorem} Let $N\geq2$, $s\in(\frac{1}{2},1)$, $\lambda_2>0$, $\frac{2s}{N}<p_2\leq \frac{2s}{N-2s}$ and $p_2<2s$. Suppose that $u\in C([0,T^*),H^{2s})$ is a radial solution of \eqref{e}. Then the solution $u(t)$
blows up in finite time in the sense that $T^*<\infty$ must hold in each of the following three cases:

1) $\lambda_1>0$, $\frac{2s}{N}<p_1<p_2$, and $E(u_0)<0$;

2) $\lambda_1<0$, $0<p_1<p_2$, and $E(u_0)<0$;

3) $\lambda_1>0$, $0<p_1\leq \frac{2s}{N}$, and $E(u_0)+CM(u_0)<0$ for some suitably large constant
C.
\end{theorem}

\begin{proof}
In what follows, we will show that the first derivative of $\mathcal{M}_{\varphi}[u(t)]$ is negative for positive times $t$. More precisely, in each of the three cases described
in Theorem 3.1, we will show that
\begin{equation}\label{301}
\frac{d}{dt}\mathcal{M}_{\varphi}[u(t)]\leq -c\|(-\Delta)^{\frac{s}{2}}u(t)\|^{2}_{L^{2}}<0
\end{equation}
for a small positive constant $c$. This implies that the solution $u(t)$
blows up in finite time. Indeed, suppose that $u(t)$ exists for all times $t\geq 0$, i.e., we can take $T^*=\infty$.

Firstly, we claim the lower bound
 \begin{equation} \label{302}
   \|(-\Delta)^{\frac{s}{2}}u(t)\|_{L^{2}}\geq C\,\,\,\mbox{for}\,\,\,t\geq0.
\end{equation}
Indeed, if this conclusion does not hold, then there exists some sequence of time $t_{k}\in[0,\infty)$ such that $\|(-\Delta)^{\frac{s}{2}}u(t_{k})\|_{L^{2}}\rightarrow0$. However, by $L^{2}$-mass conservation and the sharp Gagliardo-Nirenberg inequality \eqref{gn}, this implies that $\int_{\mathbb{R}^{N}}|u(t_k,x)|^{2p_1+2}dx\rightarrow0$ and $\int_{\mathbb{R}^{N}}|u(t_k,x)|^{2p_2+2}dx\rightarrow0$ as well. Hence, we have $E(u(t_{k}))\rightarrow0$, which is a contradiction to $E(u(t_{k}))=E(u_0)<0$. Thus, we deduce that \eqref{302} holds.

Next,
 it follows from \eqref{301} and \eqref{302} that $\frac{d}{dt}\mathcal{M}_{\varphi}[u(t)]\leq -C$ with some constant $C>0$. Integrating this bound, we conclude that $\mathcal{M}_{\varphi}[u(t)]<0$ for all $t\geq t_1$ with some time sufficiently large time $t_{1}\gg1$. Thus, integrating \eqref{301} on $[t_{1},t]$, we obtain
\begin{equation} \label{302y}
\mathcal{M}_{\varphi}[u(t)]\leq-c\int_{t_{1}}^{t}\|(-\Delta)^{\frac{s}{2}}u(\tau)\|^{2}_{L^{2}}d\tau\,\,\,\,\mbox{for all}~~~t\geq t_{1}.
\end{equation}
On the other hand, we use Lemma 2.4 and $L^{2}$-mass conservation to find that
\begin{equation} \label{302x}
\mid\mathcal{M}_{\varphi}[u(t)]\mid\leq C(\varphi_{R})(\|(-\Delta)^{\frac{s}{2}}u(t)\|^{\frac{1}{s}}_{L^{2}}
+\|(-\Delta)^{\frac{s}{2}}u(t)\|^{\frac{1}{2s}}_{L^{2}}),
\end{equation}
where we used the interpolation estimate $\| |\nabla|^{\frac{1}{2}} u\|_{L^{2}}\leq\| u\|_{L^{2}}^{1-\frac{1}{2s}}\|(-\Delta)^{\frac{s}{2}}u\|_{L^{2}}^{\frac{1}{2s}}$ for $s>\frac{1}{2}$.

So, we deduce from \eqref{302} and \eqref{302x} that
 \begin{equation}
 \begin{gathered}
   |\mathcal{M}_{\varphi}[u(t)]|\leq C(\varphi_{R})\|(-\Delta)^{\frac{s}{2}}u(t)\|^{\frac{1}{s}}_{L^{2}}.
   \end{gathered}
 \label{e3.11}
\end{equation}
This, together with \eqref{302y}, implies that
\begin{equation}
 \begin{gathered}
  \mathcal{M}_{\varphi}[u(t)]\leq -C(\varphi_{R})\int_{t_{1}}^{t} |\mathcal{M}_{\varphi}[u(\tau)]|^{2s}d\tau\,\,\,\,\,\mbox{for}\,\,\,\,t\geq t_{1}.
   \end{gathered}
 \label{e3.12}
\end{equation}
This yields $\mathcal{M}_{\varphi}[u(t)]\leq-C(\varphi_{R})| t-t_{\ast}|^{1-2s}$ for $s>\frac{1}{2}$ with some $t_{\ast}<+\infty$. Therefore, we have $\mathcal{M}_{\varphi}[u(t)]\rightarrow-\infty$ as $t\rightarrow t_{\ast}$. hence the solution $u(t)$ cannot exist for all time $t\geq0$ and consequently we must have that $T^*<+\infty$ holds.

For the remainder of the proof, we will derive \eqref{301} in each of the three cases
described in Theorem 3.1.

Case 1): $\lambda_1>0$, $\frac{2s}{N}<p_1<p_2$, and $E(u_0)<0$.

By  the conservation of energy, and our assumptions, \eqref{vir} with $\varepsilon_1$ and
$\varepsilon_2$ sufficiently small and fixed, we deduce the inequality (with $\circ_R(1)\rightarrow 0$ as $R\rightarrow \infty$ uniformly in $t$)
\begin{align}\label{c5x}
 \frac{d}{dt}\mathcal{M}_{\varphi}[u(t)]\leq&4p_1NE(u_0)-2(p_1 N-2s)\|(-\Delta)^{\frac{s}{2}}u(t)\|^{2}_{L^{2}}
 +\frac{2\lambda_2N(p_1-p_2)}{p_2+1}\|u(t)\|_{2p_2+2}^{2p_2+2}\nonumber\\
 &+\circ_R(1)(1+\|(-\Delta)^{\frac{s}{2}}u(t)\|^{\frac{p_1}{s}+\varepsilon_1}_{L^{2}}
 +\|(-\Delta)^{\frac{s}{2}}u(t)\|^{\frac{p_2}{s}+\varepsilon_2}_{L^{2}})\nonumber\\
 \leq &-(p_1 N-2s)\|(-\Delta)^{\frac{s}{2}}u(t)\|^{2}_{L^{2}}~~for~all~t\in[0,T^*),
\end{align}
provided that $R\gg 1$ is taken sufficiently large. In the last step, we use $E(u_0)<0$, Young's inequality, $\frac{p_1}{s}+\varepsilon_1<2$ and $\frac{p_2}{s}+\varepsilon_2<2$ when $\varepsilon_1$ and $\varepsilon_2$ are sufficiently small. Hence, \eqref{301} holds with $c=p_2 N-2s$.

Case 2): $\lambda_1<0$, $0<p_1<p_2$, and $E(u_0)<0$.

In this case, by a similar argument as \eqref{c5x}, we obtain
\begin{align}\label{c5y}
\frac{d}{dt}\mathcal{M}_{\varphi}[u(t)]\leq&4p_2NE(u_{0})-2(p_2 N-2s)\|(-\Delta)^{\frac{s}{2}}u(t)\|^{2}_{L^{2}}+\frac{2\lambda_1N(p_2-p_1)}
{p_1+1}\|u(t)\|_{2p_1+2}^{2p_1+2}\nonumber\\
 &+\circ_R(1)(1+\|(-\Delta)^{\frac{s}{2}}u(t)\|^{\frac{p_1}{s}+\varepsilon_1}_{L^{2}}
 +\|(-\Delta)^{\frac{s}{2}}u(t)\|^{\frac{p_2}{s}+\varepsilon_2}_{L^{2}})\nonumber\\
 \leq&-(p_2 N-2s)\|(-\Delta)^{\frac{s}{2}}u(t)\|^{2}_{L^{2}}~~for~all~t\in[0,T^*),
\end{align}
provided that $R\gg 1$ is taken sufficiently large. This implies \eqref{301} with $c=p_2 N-2s$.

Case 3): $\lambda_1>0$, $0<p_1\leq \frac{2s}{N}$, and $E(u_0)+CM(u_0)<0$ for some suitably large constant
$C$.

As $p_2>\frac{2s}{N}$, we can find a small constant $\varepsilon$ such that $p_2>\frac{2s+\varepsilon}{N}$. It is immediate that $\theta:=\frac{2s+\varepsilon}{p_2N}<1$. Therefore, by  the conservation of energy, and our assumptions, \eqref{vir} with $\varepsilon_1$ and
$\varepsilon_2$ sufficiently small and fixed, we deduce the inequality (with $\circ_R(1)\rightarrow 0$ as $R\rightarrow \infty$ uniformly in $t$)
\begin{align}\label{c5z}
 &\frac{d}{dt}\mathcal{M}_{\varphi}[u(t)]\nonumber\\\leq&
 4s\|(-\Delta)^{\frac{s}{2}}u(t)\|^{2}_{L^{2}}
 -\frac{2\lambda_1Np_1}{p_1+1}\|u(t)\|_{2p_1+2}^{2p_1+2}-\frac{2\lambda_2Np_2}{p_2+1}
 \|u(t)\|_{2p_2+2}^{2p_2+2}\nonumber\\
 &+\circ_R(1)(1+\|(-\Delta)^{\frac{s}{2}}u(t)\|^{\frac{p_1}{s}+\varepsilon_1}_{L^{2}}
 +\|(-\Delta)^{\frac{s}{2}}u(t)\|^{\frac{p_2}{s}+\varepsilon_2}_{L^{2}})\nonumber\\=& 4s\|(-\Delta)^{\frac{s}{2}}u(t)\|^{2}_{L^{2}}
 -\frac{2\lambda_1Np_1}{p_1+1}\|u(t)\|_{2p_1+2}^{2p_1+2}
 -\frac{2\lambda_2Np_2\theta}{p_2+1}\|u(t)\|_{2p_2+2}^{2p_2+2}
 \nonumber\\
 &-\frac{2\lambda_2Np_2(1-\theta)}{p_2+1}\|u(t)\|_{2p_2+2}^{2p_2+2}+\circ_R(1)(1+\|(-\Delta)^{\frac{s}{2}}u(t)\|^{\frac{p_1}{s}+\varepsilon_1}_{L^{2}}
 +\|(-\Delta)^{\frac{s}{2}}u(t)\|^{\frac{p_2}{s}+\varepsilon_2}_{L^{2}})\nonumber\\\leq& 4s\|(-\Delta)^{\frac{s}{2}}u(t)\|^{2}_{L^{2}} +2Np_2\theta\left(2E-\|(-\Delta)^{\frac{s}{2}}u(t)\|^{2}_{L^{2}}
 +\frac{\lambda_1}{p_1+1}\|u(t)\|_{2p_1+2}^{2p_1+2}\right)
 \nonumber\\
 &-\frac{2\lambda_1Np_1\theta}{p_1+1}\|u(t)\|_{2p_1+2}^{2p_1+2}
-\frac{2\lambda_2Np_2(1-\theta)}{p_2+1}\|u(t)\|_{2p_2+2}^{2p_2+2}\nonumber\\
 &+\circ_R(1)(1+\|(-\Delta)^{\frac{s}{2}}u(t)\|^{\frac{p_1}{s}+\varepsilon_1}_{L^{2}}
 +\|(-\Delta)^{\frac{s}{2}}u(t)\|^{\frac{p_2}{s}+\varepsilon_2}_{L^{2}})\nonumber\\\leq& -(Np_2\theta-2s)\|(-\Delta)^{\frac{s}{2}}u(t)\|^{2}_{L^{2}} +4Np_2\theta E+\frac{2N\theta\lambda_1(p_2-p_1)}{p_1+1}\|u(t)\|_{2p_1+2}^{2p_1+2}
 \nonumber\\
 &
-\frac{2\lambda_2Np_2(1-\theta)}{p_2+1}\|u(t)\|_{2p_2+2}^{2p_2+2}~~for~all~t\in[0,T^*),
\end{align}
provided that $R\gg 1$ is taken sufficiently large.

By Young's inequality, for any positive constants $a$ and $\delta$,
\[
a^{2p_1+2}\leq C(\delta)a^2+\delta a^{2p_2+2}.
\]
Hence,
\begin{align*}
&\frac{2N\theta\lambda_1(p_2-p_1)}{p_1+1}\|u(t)\|_{L^{2p_1+2}}^{2p_1+2}\nonumber \\\leq & C(\delta)\frac{2N\theta\lambda_1(p_2-p_1)}{p_1+1}\|u(t)\|_{L^2}^{2}+\delta \frac{2N\theta\lambda_1(p_2-p_1)}{p_1+1}\|u(t)\|_{L^{2p_2+2}}^{2p_2+2}.
\end{align*}
Choosing $\delta>0$ sufficiently small such that
\[
\delta \frac{2N\theta\lambda_1(p_2-p_1)}{p_1+1}<\frac{2\lambda_2Np_2(1-\theta)}{p_2+1},
\]
we obtain
\begin{align}\label{c5}
\frac{d}{dt}\mathcal{M}_{\varphi}[u(t)]\leq&
-(Np_2\theta-2s)\|(-\Delta)^{\frac{s}{2}}u(t)\|^{2}_{L^{2}}
\nonumber \\&+C(\delta)\frac{2N\theta\lambda_1(p_2-p_1)}{p_1+1}\|u(t)\|_{L^2}^{2} +4Np_2\theta E,
\end{align}
which, as long as
\[
C(\delta)\frac{2N\theta\lambda_1(p_2-p_1)}{p_1+1}\|u(t)\|_{L^2}^{2} +4Np_2\theta E<0,
\]
yields
\[
 \frac{d}{dt}\mathcal{M}_{\varphi}[u(t)]\leq
-(Np_2\theta-2s)\|(-\Delta)^{\frac{s}{2}}u(t)\|^{2}_{L^{2}}.
\]
This proves \eqref{301} in this case.
\end{proof}

 According to the local well-posedness theory of the fractional nonlinear Schr\"{o}dinger equation and Theorem 3.1, the solution of \eqref{e} with small initial data exists globally, and for some large initial data, the solution may blow up in finite time. Thus, whether there exists a sharp threshold of blow-up and global existence for \eqref{e} is of particular interest. On the other hand, the following problems are very important
from the view-point of physics. Under what
conditions will the condensate become unstable to collapse (blow-up)? And under what
conditions will the condensate be exist for all time (global existence)? Especially the
sharp thresholds for blow-up and global existence are pursued strongly (see \cite{ca2003,frwa,ss,we,zj,zz} and their references).
 For equation \eqref{e}, there are two nonlinearities
and there are no scaling invariance, which are the main difficulties. We obtain the following sharp conditions of blow-up and global existence for \eqref{e}
by constructing some new estimates.

\begin{theorem} Let $N\geq2$, $s\in(\frac{1}{2},1)$, $\lambda_1=\lambda_2=1$, $\frac{2s}{N}\leq p_1<p_2< \frac{2s}{N-2s}$ and $p_2<2s$. Suppose that $u\in C([0,T^*),H^{2s})$ is a radial solution of \eqref{e}. Then we have the following sharp criteria of blow-up and global existence for \eqref{e}.

1) $p_1=\frac{2s}{N}$. Let $\|u_0\|_{L^2}<\|Q_1\|_{L^2}$ and $E(u_0)<h(y_0)$. If $\|(-\Delta)^{s/2} u_0\|<y_0$, then the solution $u(t)$ of \eqref{e} exists globally; If $\|(-\Delta)^{s/2} u_0\|>y_0$, then the solution $u(t)$ of \eqref{e} blows up in finite time in the sense that $T^*<\infty$ must hold, where $Q_1$ is the ground state solution of
\eqref{ell} with $p$ replaced by $p_1$, $y_0$ and $h(y_0)$ are defined by \eqref{ee3''} and \eqref{ee3} respectively.

2) $p_1>\frac{2s}{N}$. Let $E(u_0)<\frac{p_1 N-2s}{2p_1 N}y_1^2$. If $\|(-\Delta)^{s/2} u_0\|<y_1$, then the solution $u(t)$ of \eqref{e} exists globally; If $\|(-\Delta)^{s/2} u_0\|>y_1$, then the solution $u(t)$ of \eqref{e} blows up in finite time in the sense that $T^*<\infty$ must hold, where $y_1$ is the unique positive solution of the equation $f(y)=0$ and $f(y)$ is defined in \eqref{ee288}.

\end{theorem}

\begin{proof}

Case 1): $p_1=\frac{2s}{N}$. Applying the sharp Gagliardo-Nirenberg inequality \eqref{gn}, we have
\begin{align}\label{ee1}
E(u(t))\geq& \frac{1}{2}\|(-\Delta)^{\frac{s}{2}}u(t)\|^2_{L^{2}}-\frac{C_1}{2p_1+2}\|(-\Delta)^{\frac{s}{2}}u(t)\|^2_{L^{2}}
\| u(t) \|^{2p_1}_{L^{2}}\nonumber\\ &-\frac{C_2}{2p_2+2}\|(-\Delta)^{\frac{s}{2}}u(t)\|^{\frac{p_2 N}{s}}_{L^{2}}
\| u(t) \|^{(2p_2+2)-\frac{p_2N}{s}}_{L^{2}},
\end{align}
where $C_1$ and $C_2$ are the optimal constants in \eqref{gn} with $p_1$ and $p_2$, respectively.

Now, we define a function $h(y)$ on $ [0,\infty)$
 by
\[
h(y)=\frac{1}{2}y^2-\frac{C_1}{2p_1+2}\| u_0 \|^{2p_1}_{L^{2}}y^2-\frac{C_2}{2p_2+2}
\| u_0 \|^{(2p_2+2)-\frac{p_2N}{s}}_{L^{2}}y^{\frac{p_2 N}{s}}.
\]
Thus, \eqref{ee1} can be expressed by $E(u(t))\geq h(\|(-\Delta)^{\frac{s}{2}}u(t)\|_{L^2})$, $h(y)$ is continuous on $ [0,\infty)$ and
\begin{equation}\label{ee2}
h'(y)=\left(1-\frac{C_1}{p_1+1}\| u_0\|^{2p_1}_{L^{2}}\right)y-\frac{C_2}{2p_2+2}\frac{p_2 N}{s}
\| u_0\|^{(2p_2+2)-\frac{p_2N}{s}}_{L^{2}}y^{\frac{p_2 N}{s}-1}.
\end{equation}
By the assumption $\|u_0\|_{L^2}<\|Q_1\|_{L^2}$, equation $h'(y)=0$ has only a positive root:
\begin{equation}\label{ee3''}
y_0=\left(\frac{1-\frac{C_1}{p_1+1}\| u_0\|^{2p_1}_{L^{2}}}{\frac{p_2 N}{s}\frac{C_2}{2p_2+2}
\|u_0\|^{(2p_2+2)-\frac{p_2N}{s}}_{L^{2}}}\right)^{\frac{s}{p_2 N-2s}}.
\end{equation}
Thus, $h(y)$ is increasing on the interval $[0,y_0)$, decreasing on the interval $[y_0,\infty)$ and
\begin{equation}\label{ee3}
h_{max}=h(y_0)=\frac{Np_2-2s}{2Np_2}\left(1-\frac{C_1}{p_1+1}\| u_0\|^{2p_1}_{L^{2}}\right)y_0^2.
\end{equation}
By the conservation of energy and $E(u_0)<h(y_0)$, we have
\begin{equation}\label{ee4}
 h(\|(-\Delta)^{\frac{s}{2}}u(t)\|_{L^2})\leq E(u(t))=E(u_0)<h(y_0),~~for~all~t\in [0,T^*).
\end{equation}

Now,
we claim that if $\|(-\Delta)^{\frac{s}{2}}u_0\|_{L^2}<y_0$, then $\|(-\Delta)^{\frac{s}{2}}u(t)\|_{L^2}<y_0$, for all $t\in [0,T^*)$. This
implies the solution $u(t)$ of \eqref{e} exists globally. We prove this result by contradiction as follows. If this conclusion does not hold, by the continuity of $\|(-\Delta)^{\frac{s}{2}}u(t)\|_{L^2}$,
there exists $t_0\in [0,T^*)$ such that $\|(-\Delta)^{\frac{s}{2}}u(t_0)\|_{L^2}=y_0$. Thus, $h(\|(-\Delta)^{\frac{s}{2}}u(t_0)\|_{L^2})=h(y_0)=h_{max}$.
Moreover, taking $t=t_0$ in \eqref{ee4}, one sees that
\[
h(\|(-\Delta)^{\frac{s}{2}}u(t_0)\|_{L^2})=h(y_0)=h_{max}\leq E(u)=E(u_0)<h_{max}.
\]
Thus the contradiction has been produced, the solution $u(t)$ of \eqref{e} exists globally.

On the other
hand, if $\|(-\Delta)^{\frac{s}{2}}u_0\|_{L^2}>y_0$, by the same argument, it follows that $\|(-\Delta)^{\frac{s}{2}}u(t)\|_{L^2}>y_0$ for all $t\in [0,T^*)$.

Next, we pick $\eta>0$ sufficiently small such that
\[
E(u_0)\leq \frac{Np_2-2s}{2Np_2}\left(1-\eta-\frac{C_1}{p_1+1}\| u_0\|^{2p_1}_{L^{2}}\right)y_0^2.
\]
Thus, by the conservation of energy, \eqref{vir} and \eqref{gn}, we deduce that
 \begin{align}\label{vir123}
 \frac{d}{dt}\mathcal{M}_{\varphi_R}[u(t)]\leq&
 4s\|(-\Delta)^{\frac{s}{2}}u(t)\|^{2}_{L^{2}}
 -\frac{2Np_1}{p_1+1}\|u(t)\|_{2p_1+2}^{2p_1+2}-\frac{2Np_2}{p_2+1}
 \|u(t)\|_{2p_2+2}^{2p_2+2}\nonumber\\
 +&\circ_R(1)(1+\|(-\Delta)^{\frac{s}{2}}u(t)\|^{\frac{p_1}{s}+\varepsilon_1}_{L^{2}}
 +\|(-\Delta)^{\frac{s}{2}}u(t)\|^{\frac{p_2}{s}+\varepsilon_2}_{L^{2}})\nonumber\\=&
 4p_2NE(u_{0})-2(p_2 N-2s)\|(-\Delta)^{\frac{s}{2}}u(t)\|^{2}_{L^{2}}
 +\frac{2(Np_2-2s)}{p_1+1}\|u(t)\|_{2p_1+2}^{2p_1+2}\nonumber\\
 +&\circ_R(1)(1+\|(-\Delta)^{\frac{s}{2}}u(t)\|^{\frac{p_1}{s}+\varepsilon_1}_{L^{2}}
 +\|(-\Delta)^{\frac{s}{2}}u(t)\|^{\frac{p_2}{s}+\varepsilon_2}_{L^{2}})\nonumber \\\leq&-(\delta\eta+\circ_R(1))\|(-\Delta)^{\frac{s}{2}}u(t)\|^{2}_{L^{2}}+\circ_R(1),
  \end{align}
 with $\circ_R(1)\rightarrow 0$ as $R\rightarrow \infty$ uniformly in $t\in [0,T^*)$, where $\delta=2(p_2 N-2s)$ and we have chosen $\varepsilon_1$ and $\varepsilon_2$ small enough such that $\frac{p_1}{s}+\varepsilon_1<2$ and $\frac{p_2}{s}+\varepsilon_2<2$.
 We thus conclude
\begin{equation}\label{ee222xy}
\frac{d}{dt}\mathcal{M}_{\varphi}[u(t)]\leq -\frac{\delta\eta}{2}\|(-\Delta)^{\frac{s}{2}}u(t)\|^{2}_{L^{2}},~~for~all~t\in[0,T^*).
\end{equation}
Suppose now that $T^*=\infty$ holds. Since $\|(-\Delta)^{\frac{s}{2}}u(t)\|_{L^{2}}>y_0>0$ for all $t\geq0$, we see from \eqref{ee222xy} that $\mathcal{M}_{\varphi}[u(t)]<0$ for all $t\geq t_1$ with some sufficiently large time $t_1\gg 1$.
Hence, by integrating on $[t_1,t]$, we obtain
\begin{equation}\label{ee233}
\mathcal{M}_{\varphi}[u(t)]\leq -\frac{\delta\eta}{2}\int_{t_1}^t\|(-\Delta)^{\frac{s}{2}}u(s)\|^{2}_{L^{2}}ds\leq0,~~for~all~t\geq t_1.
\end{equation}
By following exactly the steps after \eqref{302y} above, we deduce that $u(t)$ cannot exist for all times $t\geq 0$ and consequently we must have that $T^*<\infty$ holds.

Case 2):
We define a function $g(y)$ on $ [0,\infty)$ by
\[
g(y)=\frac{1}{2}y^2-\frac{C_1}{2p_1+2}\|u_0\|^{(2p_1+2)-\frac{p_1N}{s}}_{L^{2}}y^{\frac{p_1 N}{s}}-\frac{C_2}{2p_2+2}
\| u_0 \|^{(2p_2+2)-\frac{p_2N}{s}}_{L^{2}}y^{\frac{p_2 N}{s}},~~y\in [0,\infty).
\]
Thus, \eqref{ee1} can be expressed by $E(u(t))\geq g(\|(-\Delta)^{\frac{s}{2}}u(t)\|_{L^2})$, $g(y)$
is continuous on $ [0,\infty)$ and
\begin{align}\label{ee288}
g'(y)&=\left(1-\frac{C_1}{2p_1+2}\frac{p_1 N}{s}\|u_0\|^{(2p_1+2)-\frac{p_1N}{s}}_{L^{2}}y^{\frac{p_1 N}{s}-2}-\frac{C_2}{2p_2+2}\frac{p_2 N}{s}
\| u_0 \|^{(2p_2+2)-\frac{p_2N}{s}}_{L^{2}}y^{\frac{p_2 N}{s}-2}\right)y\nonumber\\&:=f(y)y.
\end{align}

For the equation $f(y)=0$, there is a unique positive solution $y_1$. Indeed, by assumption $\frac{2s}{N}<p_1<p_2<\frac{2s}{N-2s}$, for $y>0$, we have
\begin{align}\label{ee2}
f'(y)=&-\frac{C_1}{2p_1+2}\frac{p_1 N}{s}(\frac{p_1 N}{s}-2)\|u_0\|^{(2p_1+2)-\frac{p_1N}{s}}_{L^{2}}y^{\frac{p_1 N}{s}-3}\nonumber\\&-\frac{C_2}{2p_2+2}\frac{p_2 N}{s}(\frac{p_2 N}{s}-2)
\| u_0\|^{(2p_2+2)-\frac{p_2N}{s}}_{L^{2}}y^{\frac{p_2 N}{s}-3}<0,
\end{align}
which implies that $f(y)$ is decreasing on $[0,\infty)$. Due to $f(0)=1$,
there exists a unique $y_1>0$ such that $f(y_1)=0$. This implies
\[
g(y_1)=\left(\frac{1}{2}-\frac{s}{p_1 N}\right)y_1^2+\frac{C_2p_2}{2p_2+2}\left(\frac{1}{p_1}-\frac{1}{p_2}\right)
\| u_0\|^{(2p_2+2)-\frac{p_2N}{s}}_{L^{2}}y_1^{\frac{p_2 N}{s}}.
\]

On the other hand, we deduce from the conservation of energy and the assumption $E(u_0)<\frac{p_1 N-2s}{2p_1 N}y_1^2$ that
\begin{align}\label{ee2}
&g(\|(-\Delta)^{\frac{s}{2}}u(t)\|_{L^2})\leq E(u(t))=E(u_0)\nonumber\\\leq & \left(\frac{1}{2}-\frac{s}{p_1 N}\right)y_1^2+C_2p_2\left(\frac{1}{p_1}-\frac{N}{s}\right)
\| u_0\|^{(2p_2+2)-\frac{p_2N}{s}}_{L^{2}}y_1^{\frac{p_2 N}{s}}=g(y_1).
\end{align}
By the same argument as Case 1), we can obtain that if $\|(-\Delta)^{\frac{s}{2}}u_0\|_{L^{2}}<y_1$, then for all $t\in [0,T^*)$,  $\|(-\Delta)^{\frac{s}{2}}u(t)\|_{L^{2}}<y_1$, which implies the solution $u(t)$ of \eqref{e} exists globally.

And if $\|(-\Delta)^{\frac{s}{2}}u_0\|_{L^{2}}>y_1$, by the same way, it follows that $\|(-\Delta)^{\frac{s}{2}}u(t)\|_{L^{2}}>y_1$ for all $t\in [0,T^*)$.

Next, we pick $\eta>0$ sufficiently small such that
\[
E(u_0)\leq (1-\eta)\frac{p_1 N-2s}{2p_1 N}y_1^2<(1-\eta)\frac{p_1 N-2s}{2p_1 N}\|(-\Delta)^{\frac{s}{2}}u(t)\|_{L^{2}}^2~~ for~ all~t\in [0,T^*).
\]
Inserting this bound into the differential inequality \eqref{vir}, we obtain
\begin{align}\label{c5}
 \frac{d}{dt}\mathcal{M}_{\varphi}[u(t)]\leq &4p_1NE(u_{0})-2(p_1 N-2s)\|(-\Delta)^{\frac{s}{2}}u(t)\|^{2}_{L^{2}}
 +\frac{2N(p_1-p_2)}{p_2+1}\|u(t)\|_{2p_2+2}^{2p_1+2}\nonumber\\
 +&\circ_R(1)(1+\|(-\Delta)^{\frac{s}{2}}u(t)\|^{\frac{p_2}{s}+\varepsilon_1}_{L^{2}}
 +\|(-\Delta)^{\frac{s}{2}}u(t)\|^{\frac{p_1}{s}+\varepsilon_2}_{L^{2}}) \nonumber \\\leq&-(\delta\eta+\circ_R(1))\|(-\Delta)^{\frac{s}{2}}u(t)\|^{2}_{L^{2}}+\circ_R(1),
  \end{align}
with $\delta=p_1 N-2s$ and $\circ_R(1)\rightarrow 0$ as $R\rightarrow \infty$ uniformly in $t$. We thus conclude
\begin{equation*}\label{ee222}
\frac{d}{dt}\mathcal{M}_{\varphi}[u(t)]\leq -\frac{\delta\eta}{2}\|(-\Delta)^{\frac{s}{2}}u(t)\|^{2}_{L^{2}},~~for~all~t\in[0,T^*).
\end{equation*}
Therefore, by the same argument as Case 1), we can obtain the desired result.
%Suppose now that $T^*=\infty$ holds. Since $\|(-\Delta)^{\frac{s}{2}}u(t)\|_{L^{2}}>y_1>0$ for all $t\geq0$, we see from \eqref{ee222} that $\mathcal{M}_{\varphi}[u(t)]<0$ for all $t\geq t_1$ with some sufficiently large time $t_1\gg 1$.
%Hence, by integrating on $[t_1,t]$, we obtain
%\begin{equation}\label{ee233}
%\mathcal{M}_{\varphi}[u(t)]\leq -\frac{\delta\eta}{2}\int_{t_1}^t\|(-\Delta)^{\frac{s}{2}}u(s)\|^{2}_{L^{2}}ds\leq0,~~for~all~t\geq t_1.
%\end{equation}
%By following exactly the steps after \eqref{302y} above, we deduce that $u(t)$ cannot exist for all times $t\geq 0$.
\end{proof}

%To solve this problem for \eqref{e}, there exists two major difficulties. One is the loss of scaling invariance for \eqref{e}; the other is that the second order derivative of $J(t)=\int |xu(t,x)|^2dx$ is in the following form:
%\begin{equation*}\label{l2}
%J''(t)=16E(u_0)+\frac{4Np_2-16}{p_2+2}\int_{\mathbb{R}^N} |u(t,x)|^{p_2+2}dx.
%\end{equation*}
%Because $\int_{\mathbb{R}^N} |u(t,x)|^{p_2+2}dx$ is a positive uncertain function, which may be unbounded with respect to time $t$, it is hard to choose $E(u_0)$ to ensure the existence of blow-up solutions.

When $0<p_1<\frac{2s}{N}$ and $p_2=\frac{2s}{N}$, the existence of blow-up solutions of \eqref{e} has not been proved yet.
In the following,  by using the scaling argument and the variational characteristic provided by the
sharp Gagliardo-Nirenberg inequality \eqref{gn}, we prove the existence of blow-up solutions for \eqref{e} and find the sharp threshold mass of blow-up and global existence for \eqref{e}.

\begin{theorem}
Let $u_0\in H^s_{rd}$, $N\geq2$, $s\in(\frac{1}{2},1)$, $\lambda_1=-1$, $\lambda_2=1$, $0<p_1<\frac{2s}{N}$  and $p_2=\frac{2s}{N}$. Assume that $Q$ is the ground state solution of \eqref{ell} with  $p=\frac{2s}{N}$. Then, we have the following sharp threshold mass of blow-up and global existence.

(i) If $\|u_0\|_{L^2}<\|Q\|_{L^2}$, then the solution of \eqref{e} exists
globally.

(ii) If the initial data $u_0=c\rho^{\frac{N}{2}} Q(\rho x)$, where the complex number
$c$ satisfying $|c|\geq 1$, and the real number $\rho >0$, then the
 solution $u$ of \eqref{e} with initial data $u_0$ blows up in finite time $0<T^*<\infty$, or $u(t)$ blows up in infinite time such that
 \[
 \| (-\Delta)^{\frac{s}{2}}u(t)\|_{L^{2}}\geq Ct^s\,\,\,\,\mbox{for all}\,\,\,\,t\geq t_{\ast},
 \]
 with some constants $C>0$ and $t_{\ast}>0$ that depend only on $u_{0},s,N$.
\end{theorem}
\textbf{Remark.} As far as we know, this result has not been proved when $\lambda_1=0$.
However, our method can be easily applied to the case of $\lambda_1=0$.
Therefore, this result is new even for \eqref{e} with $\lambda_1=0$.
\begin{proof}
(i) We deduce from the energy conservation \eqref{h} and the sharp Gagliardo-Nirenberg inequality\eqref{gn} that for all $t\in[0,T^*)$
\begin{align*}
E(u(t)) =&\frac{1}{2}\int_{\mathbb{R}^N} |(-\Delta)^{s/2} u(t,x)|^2 dx+\frac{1}{2p_1+2}\int_{\mathbb{R}^N}
|u(t,x)|^{2p_1+2}dx \nonumber\\&-\frac{1}{2p_2+2}\int_{\mathbb{R}^N}
|u(t,x)|^{2p_2+2}dx\nonumber\\ \geq & \left(\frac{1}{2}-\frac{\|u_0\|_{L^2}^{2p_2}}{2\|Q\|_{L^2}^{2p_2}}\right)\|(-\Delta)^{s/2} u(t)\|_{L^2}^2.
\end{align*}
From the hypothesis $\|u_0\|_{L^2}<\|Q\|_{L^2}$, there exists a constant $C>0$ such that
 $E(u_0)=E(u(t))\geq C\|(-\Delta)^{s/2} u(t)\|_{L^2}^2$ for all $t\in[0,T^*)$. Then, $u(t)$ is bounded in $H^s$ for all  $t\in[0,T^*)$ by the conservation of mass, and $u(t)$ exists globally in $H^s$ by the local well-posedness (see Proposition 2.1). This completes the proof of (i).

(ii)
 By the
definition of initial data $u_0(x)=c\rho^{\frac{N}{2}} Q(\rho x)$ and the Poho\u{z}aev identity for equation \eqref{ell}, i.e., $\|(-\Delta)^{s/2} Q\|_{L^2}^2=\frac{1}{p_2+1}\|Q\|^{2p_2+2}_{L^{2p_2+2}}$, we deduce that
\begin{align}\label{h1}
E(u_0)=&\frac{|c|^2\rho^{2s}}{2}\int_{\mathbb{R}^N} |(-\Delta)^{s/2} Q(x)|^2dx+\frac{|c|^{2p_1+2}\rho^{Np_1}}{2p_1+2}\int_{\mathbb{R}^N}
|Q(x)|^{2p_1+2}dx\nonumber \\
&-\frac{|c|^{2p_2+2}\rho^{Np_2}}{2p_2+2}\int_{\mathbb{R}^N}
|Q(x)|^{2p_2+2}dx\nonumber \\
=&-\frac{|c|^2\rho^{2s}}{2}(|c|^{2p_2}-1)\|(-\Delta)^{s/2} Q\|_{L^2}^2+\frac{|c|^{2p_1+2}\rho^{Np_1}}{2p_1+2}\int_{\mathbb{R}^N}
|Q(x)|^{2p_1+2}dx.
\end{align}
Now, taking $\rho$ such that
\[
\frac{|c|^{2p_1}\|Q\|^{2p_1+2}_{L^{2p_1+2}}}{(p_1+1)(|c|^{2p_2}-1)\|(-\Delta)^{s/2} Q\|_{L^2}^2}< \rho^{2s-
Np_1}.
\]
This implies $E(u_0) < 0$.

On the other hand, by a similar argument in \cite{bou}, we can choose $\varphi_R(r)$ and $\eta>0$ sufficiently small such that
\[
\psi_{1,R}(r)-C(\eta)(\psi_{2,R}(r))^{\frac{N}{2s}}\geq0~~for~all~r>0~and~R>0.
\]
Thus if we choose $\eta\ll 1$ sufficiently small and then $R\gg 1$ sufficiently large, we can apply Lemma 2.8 to deduce that
\begin{equation}\label{l10}
\frac{d}{dt}\mathcal{M}_{\varphi_R}[u(t)]\leq 4sE(u_{0}),~~for~all~t\in[0,T^*).
\end{equation}

Next, we suppose that $u(t)$ exists for all time $t\geq0$, i.e., $T^*=\infty$. It follows from \eqref{l10} that
\begin{equation}\label{e4.9}
\mathcal{M}_{\varphi_R}[u(t)]\leq -ct\,\,\,\,\mbox{for} \,\,\,t>t_{0},
\end{equation}
with some sufficiently large time $t_{0}>0$ and some constant $c>0$ depending only on $s$ and $E(u_0)<0$. On the other hand, if we invoke Lemma 2.4, we see that
\begin{align}\label{e4.10}
\mathcal{M}_{\varphi_R}[u(t)]&\leq C(\varphi_{R})(\||\nabla|^{\frac{1}{2}}u(t)\|^{2}_{L^{2}}+\| u(t)\|_{L^{2}}\||\nabla|^{\frac{1}{2}}u(t)\|_{L^{2}})\nonumber\\
&\leq C(\varphi_{R})(\||\nabla|^{\frac{1}{2}}u(t)\|^{2}_{L^{2}}+1)\nonumber\\
&\leq C(\varphi_{R})(\|(-\Delta)^{\frac{s}{2}}u(t)\|^{\frac{1}{s}}_{L^{2}}+1),
\end{align}
where we also used the conservation of $L^{2}$-mass together with the interpolation estimate $\| |\nabla|^{\frac{1}{2}} u\|_{L^{2}}\leq\| u\|_{L^{2}}^{1-\frac{1}{2s}}\|(-\Delta)^{\frac{s}{2}}u\|_{L^{2}}^{\frac{1}{2s}}$ for $s>\frac{1}{2}$.
Combining \eqref{e4.9} and \eqref{e4.10}, we finally get
 \[
 \| (-\Delta)^{\frac{s}{2}}u(t)\|_{L^{2}}\geq Ct^s\,\,\,\,\mbox{for all}\,\,\,\,t\geq t_{\ast},
 \]
 with some constants $C>0$ and $t_{\ast}>0$ that depend only on $u_{0},s,N$.
\end{proof}

\section{Dynamic of blow-up solutions in the case of $L^2$-critical}
In this section, we investigate some dynamical properties of blow-up solutions for \eqref{e} with $\lambda_1=-1$, $\lambda_2=1$,
$0<p_1<\frac{2s}{N}$ and $p_2=\frac{2s}{N}$. In this case, we  prove that there exists the sharp threshold
mass $\|Q\|_{L^2}$ of blow-up and global existence in Section 3. Hence, the study of the dynamical properties of blow-up
solutions around the sharp threshold mass is of particular interest. For this aim, we firstly obtain the following refined compactness result by using the profile
decomposition of bounded sequences in $H^s$ and the inequality \eqref{gn}.
\begin{lemma}
Let $N\geq2$ and $s\in(\frac{1}{2},1)$. If $\{u_n\}_{n=1}^{\infty}$ be a bounded sequence in $H^s$, such that
\begin{equation*}
\limsup_{n\rightarrow \infty}\|(-\Delta)^{s/2} u_n\|_{L^2}\leq M,~~~\limsup_{n\rightarrow \infty}\|u_n\|_{L^{4s/N+2}}\geq m>0.
\end{equation*}
Then, there exist $V\in H^s$ and $\{x_n\}_{n=1}^{\infty}\subset \mathbb{R}^N$ such that, up to a subsequence,
\begin{equation*}
u_n(\cdot+x_n)\rightharpoonup V~~weakly~in~H^s
\end{equation*}
with
\begin{equation*}
\|V\|_{L^2}^{4s/N}\geq \frac{m^{4s/N+2}N\|Q\|_{L^2}^{\frac{4s}{N}}}{(2s+N)M^2}.
\end{equation*}
where $Q$ is the ground state solution of \eqref{ell} with $p=\frac{2s}{N}$.
\end{lemma}
\begin{proof}
We deduce from the profile decomposition (Proposition 2.3) that
\begin{equation}\label{401}
u_n(x)=\sum_{j=1}^{l}V^j(x-x_n^j)+v_n^l,
\end{equation}
with $\limsup_{n\rightarrow \infty}\|v_n^l\|_{L^q}\rightarrow 0$ as $l\rightarrow \infty$.

From \eqref{401}, \eqref{gn} and Proposition 2.3, we obtain
\begin{align}\label{4111}
 m^{4s/N+2}\leq &\limsup_{n\rightarrow \infty}\int |u_n|^{4s/N+2}dx \nonumber\\\leq &
\int |\sum_{j=1}^{\infty}V^j(x-x_n^j)|^{4s/N+2}dx
\nonumber\\
  \leq &\sum_{j=1}^{\infty}\int |V^j|^{4s/N+2}dx\nonumber\\
  \leq &\sum_{j=1}^{\infty}\frac{2s+N}{N\|Q\|_{L^2}^{4s/N}}\|V^j\|_{L^2}^{4s/N}\|(-\Delta)^{s/2} V^j\|_{L^2}^2\nonumber\\
  \leq &\frac{2s+N}{N\|Q\|_{L^2}^{4s/N}}\sup \{\|V^j\|_{L^2}^{4s/N},j\geq 1\}\sum_{j=1}^{\infty}\|(-\Delta)^{s/2} V^j\|_{L^2}^2.
\end{align}

On the other hand, we observe that
\begin{equation}\label{412}
\sum_{j=1}^{\infty}\|(-\Delta)^{s/2} V^j\|_{L^2}^2\leq \limsup_{n\rightarrow \infty}\|(-\Delta)^{s/2} u_n\|_{L^2}^2\leq M^2.
\end{equation}
Therefore, it follows from \eqref{4111} and \eqref{412} that
\[
\sup \{\|V^j\|_{L^2}^{4s/N},j\geq 1\}\geq \frac{m^{4s/N+2}N\|Q\|_{L^2}^{\frac{4s}{N}}}{(2s+N)M^2}.
\]
Since the series $\sum_{j=1}^{\infty}\| V^j\|_{L^2}^2$ is convergent, there exists $j_0\geq 1$ such that
\[
 \|V^{j_0}\|_{L^2}^{\frac{4s}{N}}\geq \frac{m^{\frac{4s}{N}+2}N\|Q\|_{L^2}^{\frac{4s}{N}}}{(2s+N)M^2}.
\]
From \eqref{3.a}, a change of variables $x=x+x_n^{j_0}$ gives
\[
u_n(x+x_n^{j_0})=V^{j_0}(x)+\sum_{j\neq j_0}V^j(x+x_n^{j_0}-x_n^j)+v_n^l(x+x_n^{j_0}).
\]
Using the pairwise orthogonality of $\{x^j_n\}_{j=1}^{\infty}$, we have
\[
V^j(\cdot+x_n^{j_0}-x_n^j)\rightharpoonup 0,~~weakly~in~H^s~~for~every~j\neq j_0.
\]
Hence, we have
\[
u_n(\cdot+x_n^{j_0})\rightharpoonup V^{j_0}+\tilde{v}^l ,~~weakly~in~H^s.
\]
where $\tilde{v}^l$ denote the weak limit of $v_n^l(x+x_n^{j_0})$.
However,
\[
 \int |\tilde{v}^l|^{\frac{4s}{N}+2} dx\leq \limsup_{n\rightarrow \infty}\int |v^l_n|^{\frac{4s}{N}+2}dx\rightarrow 0.
\]
Thus, it follows from uniqueness of weak limit that $\tilde{v}^l=0$ for all $l\geq J_0$.
Therefore,
\[
u_n(\cdot+x_n^{j_0})\rightharpoonup V^{j_0},~~weakly~in~H^s.
\]
This completes the proof.
\end{proof}

By applying the refined compactness Lemma 4.1, we can obtain the following $L^2$-concentration and rate of $L^2$-concentration of blow-up solutions of \eqref{e}.

\begin{theorem}($L^2$-concentration)
Let $N\geq2$, $s\in(\frac{1}{2},1)$, $u_0\in H^s$, $\lambda_1=-1$, $\lambda_2=1$, $0<p_1<\frac{2s}{N}$ and $p_2=\frac{2s}{N}$. If the solution $u$ of \eqref{e} blows up in finite time $T^*>0$.
Let $a(t)$ be a real-valued nonnegative function defined on $[0,T^*)$ satisfying $a(t)\|(-\Delta)^{s/2} u(t)\|_{L^2}^{1/s}\rightarrow \infty $ as $t\rightarrow T^*$. Then there exists $x(t)\in \mathbb{R}^{N}$ such that
\begin{equation}\label{41}
\liminf_{t\rightarrow T^*}\int_{|x-x(t)|\leq a(t)}|u(t,x)|^2dx\geq \int_{\mathbb{R}^N} |Q(x)|^2dx.
\end{equation}
where $Q$ is the ground state solution of \eqref{ell} with $p=\frac{2s}{N}$.

%(2) If $\mathcal{A}$ is the set of weak $L^2$-limit points of $u(t)$ as $t\rightarrow T^*$, then
%\begin{equation}\label{42}
%\|V\|_{L^2}^2\leq \|u_0\|_{L^2}^2-\|R\|_{L^2}^2~~for~all~V\in \mathcal{A}.
%\end{equation}

%(4) Assume that the initial mass is small super-critical: there exists $\alpha_1>0$ such that for all $0<\alpha'\leq \alpha_1$, there exists $\delta(\alpha')>0$ with $\delta(\alpha')\rightarrow 0$ as $\alpha'\rightarrow 0$, and for any $u_0\in H^1$ and $0 <\alpha (u_0):=\|u_0\|_{L^2}^2-\|R\|_{L^2}^2<\alpha'$. Then, there exist functions $y(t)\in \mathbb{R}^N$ and $\theta(t)\in \mathbb{R}$ such that when $t\rightarrow T^*$,
%\begin{equation}\label{43}
%\|\rho^{N/2}(t)u(t,\rho(t)(\cdot+y(t)))e^{i\theta(t)}-R\|_{H^1}\leq \delta(\alpha').
%\end{equation}

\end{theorem}
\textbf{Remark.} Theorem 4.2 gives the $L^2$-concentration and rate of $L^2$-concentration of blow-up solutions of \eqref{e}. Indeed, we can choose $a(t)=\frac{1}{{\|(-\Delta)^{s/2} u(t)\|}^{\frac{1}{s}-\delta}_{L^2}}$ with $0<\delta<\frac{1}{s}$. It is obvious that $\lim_{t\rightarrow T^*} a(t)=0$ and $a(t)$ satisfies the assumption in Theorem 4.2. Applying Theorem 4.2, if $u$ is a blow-up solution of \eqref{e} and $T^*$ its blow-up time, then for every $r>0$, there exists a function $x(t)\in\mathbb{R}^N$ such that
\[
\liminf_{t\rightarrow T^*}\int_{|x-x(t)|\leq r}|u(t,x)|^2dx\geq\int_{\mathbb{R}^N}|Q(x)|^2dx.
\]
Meanwhile, it follows from the choice of $a(t)$ that for any function $0<a(t)\leq \frac{1}{{\|(-\Delta)^{s/2} u(t)\|}^{\frac{1}{s}-\delta}_{L^2}}$, \eqref{41} holds, which implies that the rate of $L^2$-concentration of blow-up solutions of \eqref{e} is $\frac{1}{{\|(-\Delta)^{s/2} u(t)\|}^{\frac{1}{s}-\delta}_{L^2}}$ with $0<\delta<\frac{1}{s}$.

\begin{proof}
 Set
\[
\rho^s(t)=\|(-\Delta)^{s/2} Q\|_{L^2}/\|(-\Delta)^{s/2} u(t)\|_{L^2}~~and~~v(t,x)=\rho^{\frac{N}{2}}(t)u(t,\rho(t) x).
\]
Let $\{t_n\}_{n=1}^\infty$ be an any time sequence such that $t_n\rightarrow T^*$, $\rho_n:=\rho(t_n)$ and $v_n(x):=v(t_n,x)$.
Then, the sequence $\{v_n\}$
satisfies
\begin{equation}\label{44}
\|v_n\|_{L^2}=\|u(t_n)\|_{L^2}=\|u_0\|_{L^2},~~\|(-\Delta)^{s/2} v_n\|_{L^2}=\rho_n^s\|(-\Delta)^{s/2} u(t_n)\|_{L^2}=\|(-\Delta)^{s/2} Q\|_{L^2}.
\end{equation}
Observe that
\begin{align}\label{45}
 H(v_n):=&\frac{1}{2}\int_{\mathbb{R}^N} |(-\Delta)^{s/2} v_n(x)|^2dx-\frac{1}{2p_2+2}\int_{\mathbb{R}^N}  |v_n(x)|^{2p_2+2}dx
\nonumber\\
  =&\rho_n^{2s}\left(\frac{1}{2}\int_{\mathbb{R}^N} |(-\Delta)^{s/2} u(t_n,x)|^2dx-\frac{1}{2p_2+2}\int_{\mathbb{R}^N}  |u(t_n,x)|^{2p_2+2}dx\right)\nonumber\\
  =&\rho_n^{2s}\left(E(u_0)-\frac{1}{2p_1+2}\int_{\mathbb{R}^N} |u(t_n,x)|^{2p_1+2}dx\right).
\end{align}
Applying the following Gagliardo-Nirenberg inequality
\begin{equation*}
\int_{\mathbb{R}^N} |u(x)|^{2p_1+2}dx\leq C\|u\|_{L^2}^{2p_1+2-\frac{Np_1}{s}}\|(-\Delta)^{s/2} u\|_{L^2}^{\frac{Np_1}{s}}\leq C\|u\|_{L^2}^{2p_1+2-\frac{Np_1}{s}}\|(-\Delta)^{s/2} u\|_{L^2}^2,
\end{equation*}
where $0<p_1<\frac{2s}{N}$.
It follows that $H(v_n)\rightarrow 0$ as $n\rightarrow\infty$,
 which implies
$\int_{\mathbb{R}^N} |v_n(x)|^{2p_2+2}dx\rightarrow \frac{2s+N}{N}\|(-\Delta)^{s/2} Q\|_{L^2}^2$.

Set $m^{2p_2+2}=\frac{2s+N}{N}\|(-\Delta)^{s/2} Q\|_{L^2}^2$ and $M=\|(-\Delta)^{s/2} Q\|_{L^2}$. Then it follows from Lemma 4.1 that there exist $V\in H^s$ and $\{x_n\}_{n=1}^\infty \subset \mathbb{R}^N$ such that, up to a subsequence,
\begin{equation}\label{46'}
v_n(\cdot +x_n)=\rho_n^{N/2}u(t_n,\rho_n(\cdot + x_n))\rightharpoonup V~~weakly~in~H^s
\end{equation}
with
\begin{equation}\label{46}
\|V\|_{L^2}\geq \|Q\|_{L^2}.
\end{equation}
Note that
\[
\frac{a(t_n)}{\rho_n}=\frac{a(t_n)\|(-\Delta)^{s/2} u(t_n)\|_{L^2}^{1/s}}{\|(-\Delta)^{s/2} Q\|_{L^2}^{1/s}}\rightarrow \infty,~~as~n\rightarrow\infty.
\]
Then for every $r>0$, there exists $n_0>0$ such that for every $n>n_0$, $r\rho_n<a(t_n)$. Therefore, using
\eqref{46'}, we obtain
\begin{align*}\label{45}
\liminf_{n\rightarrow \infty}\sup_{y\in \mathbb{R}^N}\int_{|x-y|\leq a(t_n)}|u(t_n,x)|^2dx&\geq \liminf_{n\rightarrow \infty}\sup_{y\in \mathbb{R}^N}\int_{|x-y|\leq r\rho_n}|u(t_n,x)|^2dx
\nonumber\\
  &\geq \liminf_{n\rightarrow \infty}\int_{|x-x_n|\leq r\rho_n}|u(t_n,x)|^2dx\nonumber\\
  &=\liminf_{n\rightarrow \infty}\int_{|x|\leq r}\rho_n^{N}|u(t_n,\rho_n(x+ x_n))|^2dx\nonumber\\
  &=\liminf_{n\rightarrow \infty}\int_{|x|\leq r}|v(t_n,x+ x_n)|^2dx\nonumber\\
  &\geq\liminf_{n\rightarrow \infty}\int_{|x|\leq r}|V(x)|^2dx,~~for~every~r>0,
\end{align*}
which means that
\[
\liminf_{n\rightarrow \infty}\sup_{y\in \mathbb{R}^N}\int_{|x-y|\leq a(t_n)}|u(t_n,x)|^2dx\geq\int_{ \mathbb{R}^N}|V(x)|^2dx.
\]
%
%Thus, for any $M>0$,
%\[
%\liminf_{n\rightarrow \infty}\int_{|x|\leq M}\rho_n^{N}|u(t_n,\rho_n(x+ x_n))|^2dx\geq \int_{|x|\leq M}|V(x)|^2dx.
%\]
%In view of the assumption  $a(t)\|\nabla u(t)\|_{L^2}\rightarrow \infty $, this implies immediately
%\begin{equation*}
%\liminf_{n\rightarrow \infty}\int_{|z-\rho_nx_n|\leq \rho_nM}|u(t_n,z)|^2dz\geq \int_{|x|\leq M}|V(x)|^2dx.
%\end{equation*}
Since the sequence $\{t_n\}_{n=1}^\infty$ is arbitrary, it follows from \eqref{46} that
\begin{equation}\label{45x}
\liminf_{t\rightarrow T^*}\sup_{y\in \mathbb{R}^N}\int_{|x-y|\leq a(t)}|u(t,x)|^2dx\geq\int_{ \mathbb{R}^N}|Q(x)|^2dx.
\end{equation}
Observe that for every $t\in [0,T^*)$, the function $g(y):= \int_{|x-y|\leq a (t)}|u(t,x)|^2dx$ is continuous on $y\in \mathbb{R}^N$ and $g(y)\rightarrow 0$ as $|y|\rightarrow \infty$. So there exists a function $x(t)\in \mathbb{R}^N$ such that for every $t\in [0,T^*)$
\begin{equation*}
 \sup_{y\in \mathbb{R}^N}\int_{|x-y|\leq a (t)}|u(t,x)|^2dx=\int_{|x-x(t)|\leq a(t)}|u(t,x)|^2dx.
\end{equation*}
This and \eqref{45x} yield \eqref{41}.
\end{proof}

In the following theorem, we study the limiting profile of blow-up solutions of \eqref{e}.
\begin{theorem}
Let $u_0\in H^s$, $\lambda_1=-1$, $\lambda_2=1$, $0<p_1<\frac{2s}{N}$, and $p_2=\frac{2s}{N}$.
Assume $\|u_0\|_{L^2}=\|Q\|_{L^2}$, and the corresponding solution $u$ of \eqref{e} blows up in finite time $T^*>0$, then there exist $x(t)\in \mathbb{R}^N$ and $\theta(t)\in [0,2\pi)$ such that
\begin{equation}\label{43'}
\rho^{N/2}(t)u(t,\rho(t)(\cdot+x(t)))e^{i\theta(t)}\rightarrow Q~strongly~in~H^s,~as~t\rightarrow T^*,
\end{equation}
where $\rho(t)=\frac{\|(-\Delta)^{s/2} Q\|_{L^2}}{\|(-\Delta)^{s/2} u(t)\|_{L^2}}$.

%(4) Assume that the initial mass is small super-critical: there exists $\alpha_1>0$ such that for all $0<\alpha'\leq \alpha_1$, there exists $\delta(\alpha')>0$ with $\delta(\alpha')\rightarrow 0$ as $\alpha'\rightarrow 0$, and for any $u_0\in H^1$ and $0 <\alpha (u_0):=\|u_0\|_{L^2}^2-\|R\|_{L^2}^2<\alpha'$. Then, there exist functions $y(t)\in \mathbb{R}^N$ and $\theta(t)\in \mathbb{R}$ such that when $t\rightarrow T^*$,
%\begin{equation}\label{43}
%\|\rho^{N/2}(t)u(t,\rho(t)(\cdot+y(t)))e^{i\theta(t)}-R\|_{H^1}\leq \delta(\alpha').
%\end{equation}

\end{theorem}

\begin{proof}
%We prove \eqref{43'} by contradiction. More precisely, our aim is to prove that for arbitrary $\varepsilon>0$ and any sequence $\{t_n\}^\infty_{n=1}$ such that $t_n\rightarrow T^*$ as $n\rightarrow \infty$, we have
%\begin{equation}\label{49}
%\|\rho^{N/2}_nu(t_n,\rho_n(\cdot+y_{n}))e^{i\theta_n}-Q(\cdot)\|_{H^1}<\varepsilon,
%\end{equation}
%where parameters: $\{\rho_n\}^\infty_{n=1}\subset \mathbb{R}^+,\{y_n\}^{\infty}_{n=1}\subset \mathbb{R}^N$, and $\{\theta_n\}^{\infty}_{n=1}\subset \mathbb{R}$. If not, then \eqref{49} is not true for some sequences $\{t_n\}^{\infty}_{n=1}$.

We use the notations of the proof of Theorem 4.2. Assume that
$\|u_0\|_{L^2}=\|Q\|_{L^2}$. Recall that we have verified that
$\|V\|_{L^2}\geq \|Q\|_{L^2}$ in the proof of Theorem 4.2. Whence
%Next, we will find a subsequence of $\{t_n\}^{\infty}_{n=1}$ such that \eqref{49} holds. Indeed, by a similar argument as \eqref{46'}, we can obtain
%\begin{equation}\label{49''}
%v_n(\cdot+y_n)e^{i\theta_n}\rightharpoonup V_1~ weakly ~in ~L^2~ with ~\|V_1\|_{L^2}\geq\|Q\|_{L^2}.
%\end{equation}
%By lower semi-continuity of the $L^2$ norm that
\[
\|Q\|_{L^2}\leq\|V\|_{L^2}\leq  \liminf_{n\rightarrow \infty }\|v_n\|_{L^2}=\liminf_{n\rightarrow \infty }\|u(t_n)\|_{L^2}=\|u_0\|_{L^2}=\|Q\|_{L^2},
\]
and then,
\begin{equation}\label{49'}
\lim_{n\rightarrow \infty }\|v_n\|_{L^2}=\|V\|_{L^2}=\|Q\|_{L^2},
\end{equation}
which implies
\[
v_n(\cdot+x_n)\rightarrow V~ strongly ~in ~L^2~as~n\rightarrow \infty.
\]
We infer from the inequality \eqref{gn} that
\[
\|v_n(\cdot+x_{n})-V\|^{2p_2+2}_{L^{2p_2+2}}\leq C\|v_n(\cdot+x_{n})-V\|^{p_2}_{L^2}\|(-\Delta)^{s/2} (v_n(\cdot+x_{n})
-V)\|^2_{L^2}.
\]
From $\|(-\Delta)^{s/2} v_n(\cdot+x_{n})\|_{L^2}\leq C$, we get
\[
v_n(\cdot+x_{n})\rightarrow V~~in ~L^{2p_2+2}~as~n\rightarrow \infty.
\]

Next, we will prove that $v_n(\cdot+x_{n})$ converges to $V$ strongly in $H^s$. For this aim, we estimate as follows:
\begin{align}\label{410}
0 =&\lim_{n\rightarrow \infty }H(v_n)
\nonumber\\
  =&\frac{1}{2}\int_{\mathbb{R}^N}|(-\Delta)^{s/2} Q(x)|^2dx-\frac{1}{2p_2+2}\lim_{n\rightarrow \infty}\int_{\mathbb{R}^N} |v_n(x)|^{2p_2+2}dx\nonumber\\
  =&\frac{1}{2}\int_{\mathbb{R}^N}|(-\Delta)^{s/2} Q(x)|^2dx-\frac{1}{2p_2+2}\int_{\mathbb{R}^N}|V(x)|^{2p_2+2}dx.
\end{align}
Using the inequality \eqref{gn}, we infer that
\begin{align}\label{411h}
&\frac{1}{2}\int_{\mathbb{R}^N}|(-\Delta)^{s/2} Q(x)|^2dx=\frac{1}{2p_2+2}\int_{\mathbb{R}^N}|V(x)|^{2p_2+2}dx
\nonumber\\&\leq\frac{1}{2}\frac{\| V\|_{L^2}^{p_2}}{\|Q\|_{L^2}^{p_2}}\|(-\Delta)^{s/2} V\|_{L^2}^2=\frac{1}{2}\|(-\Delta)^{s/2} V\|_{L^2}^2.
\end{align}
On the other hand, we deduce from \eqref{44} that $ \|(-\Delta)^{s/2} V\|_{L^2}\leq \liminf_{n\rightarrow \infty}\|(-\Delta)^{s/2} v_n(\cdot+x_{n})\|_{L^2}=\|(-\Delta)^{s/2} Q\|_{L^2}$. Hence, we have $\| Q\|_{H^s}=\|V\|_{H^s}$ and
\begin{equation}\label{411}
v_n(\cdot+x_{n})\rightarrow V~ strongly ~in ~H^s~as~n\rightarrow \infty.
\end{equation}
%Now, collect the properties of $V_1(x)$:
%\[
%H(V_1)=0,~~\|V_1\|_{L^2}=\|R\|_{L^2}~and~\|\nabla R\|_{L^2}=\|\nabla V_1\|_{L^2}.
%\]
This and \eqref{410} imply that
\[
H(V)=\frac{1}{2}\int_{\mathbb{R}^N}|(-\Delta)^{s/2} V(x)|^2dx-\frac{1}{2p_2+2}\int_{\mathbb{R}^N}|V(x)|^{2p_2+2}dx=0.
\]
Up to now, we have verified that
\[
\|V\|_{L^2}=\|Q\|_{L^2},~\|(-\Delta)^{s/2} V\|_{L^2}=\|(-\Delta)^{s/2} Q\|_{L^2}~ and ~H(V)=0.
\]
The variational characterization of the ground state implies that there exist $x_0\in  \mathbb{R}^N$ and $\theta\in [0,2\pi)$ such that
\[
V(x)=e^{i\theta}Q(x+x_0),
\]
and
\[
\rho^{N/2}_nu(t_n,\rho_n(\cdot+x_n))\rightarrow e^{i\theta}Q(\cdot+x_0)~strongly~in~H^s~as~n\rightarrow \infty.
\]
Since the sequence $\{t_n\}_{n=1}^\infty$ is arbitrary, we infer that there are two functions $x(t)\in \mathbb{R}^N$ and $\theta(t)\in [0,2\pi)$ such that
\[
\rho^{N/2}(t)e^{i\theta(t)}u(t,\rho(t)(x+x(t)))\rightarrow Q~strongly~in~H^s~as~t\rightarrow T^*.
\]
\end{proof}

\end{document}